\documentclass[11pt]{amsart}
\usepackage{amssymb}
\usepackage{amsmath}
\usepackage{amsthm}
\usepackage{color}
\usepackage{enumerate}
\usepackage{xypic}
\usepackage[vmargin=2cm,hmargin=2cm]{geometry}

\newtheorem{theorem}{Theorem}[section]
\newtheorem{lemma}[theorem]{Lemma}

\newtheorem{corollary}[theorem]{Corollary}

\theoremstyle{remark}
\newtheorem{remark}[theorem]{Remark}
\newtheorem{example}[theorem]{Example}

\newcommand\mut[1]{\ignorespaces}

\title{Nonhomogeneous patterns on numerical semigroups}


\author[Bras-Amor\'os]{Maria Bras-Amor\'os}
\address{Departament d'Enginyeria Inform\`atica i Matem\`atiques, Av. Pa\"isos Catalans, 26 Campus Sescelades, 43007 Tarragona, Espa\~na}
\email{maria.bras@urv.cat}

\author[Garc\'\i a-S\'anchez]{Pedro A. Garc\'\i a-S\'anchez} 
\address{Departamento de \'Algebra, Universidad de Granada, E-18071 Granada, Espa\~na}
\email{pedro@ugr.es}

\thanks{The first and third authors are supported by the Spanish Government through the projects TIN2012-32757 ``ICWT" and CONSOLIDER INGENIO 2010 CSD2007-00004 ``ARES". The second author is supported by the projects MTM2010-15595 and FQM-343,  FQM-5849, and FEDER funds.}

\author[Vico-Oton]{Albert Vico-Oton}
\address{Departament d'Enginyeria Inform\`atica i Matem\`atiques, Av. Pa\"isos Catalans, 26 Campus Sescelades, 43007 Tarragona, Espa\~na}
\email{albert.vico@urv.cat}

\date{\today}

\renewcommand\leq\leqslant
\renewcommand\le\leqslant
\renewcommand\geq\geqslant
\renewcommand\ge\geqslant

\begin{document}

\begin{abstract}
Patterns on numerical semigroups are multivariate linear polynomials, and they are said to be admissible if there exists a numerical semigroup such that evaluated at any nonincreasing sequence of elements of the semigroup gives integers belonging to the semigroup. In a first approach, only homogeneous patterns where analized. In this contribution we study conditions for an eventually non-homogeneous pattern to be admissible, and particularize this study to the case the independent term of the pattern is a multiple of the multiplicity of the semigroup. Moreover, for the so called strongly admissible patterns, the set of numerical semigroups admitting these patterns with fixed multiplicity $m$ form an $m$-variety, which allows us to represent this set in a tree and to describe minimal sets of generators of the semigroups in the variety with respect to the pattern. Furthermore, we characterize strongly admissible patterns having a finite associated tree.
\end{abstract}

\maketitle

\section{Introduction}


A numerical semigroup $\Lambda$ is a subset of the nonnegative integers ${\mathbb N}_0$ that contains $0$ and is closed under addition, and such that ${\mathbb N}_0\setminus\Lambda$ is finite. The number $\#({\mathbb N}_0\setminus\Lambda)$ is denoted the {\it genus} of the semigroup and the first nonzero nongap of $\Lambda$ is called its {\it multiplicity}. The largest integer not in $\Lambda$ is denoted $\mathrm F(\Lambda)$, and it is called the \emph{Frobenius number} of $\Lambda$.

Arf semigroups appear in many theoretical problems in algebraic geometry as well as in some applied areas such as coding theory \cite{BDF,RGGB,CaFaMu,BAO,BA:Arf,BA:Acute}. Arf semigroups are those semigroups such that for any elements $x_1,x_2,x_3$ in the semigroup with $x_1\geq x_2\geq x_3$, the integer $x_1+x_2-x_3$ also belongs to the semigroup. 

This definition inspired studying the so-called patterns on numerical semigroups \cite{BA-GS}. Patterns on numerical semigroups are multivariate polynomials such that evaluated at any decreasing sequence of elements of the semigroup give integers belonging to the semigroup.

For their simplicity, and for their inspiration in Arf semigroups, patterns were first defined to be linear and homogeneous. However, Arf semigroups are of maximal embedding dimension, and this larger class of numerical semigroups fulfills a nonhomogeneous pattern. Lately, other families of numerical semigroups that satisfy a nonhomogeneous pattern have appeared in very different areas of applied mathematics. This suggests the need for studying nonhomogeneous patterns on numerical semigroups.

In this contribution we give some results on nonhomogeneous linear patterns. We start by presenting some motivating examples. Then we focus on the problem of characterizing patterns that are admissible, that is, there is at least a nontrivial numerical semigroup admitting them. Next, we particularize this study to the case the independent term of the pattern is a multiple of the multiplicity. Moreover, if we fix the multiplicity, the set of numerical semigroups admitting a strong admissible pattern is closed under intersections and the adjoin of the Frobenius number. This motivates the definition of $m$-varieties and the concept of minimal generating system associated to an $m$-variety, which allow us to represent the elements in an $m$-variety in a tree rooted in $\{0\}\cup(m+\mathbb N_0)$. Finally, for a given multiplicity we characterize those strongly admissible patterns yielding a finite tree.

\section{Motivating examples}

We will present three different scenarios where nonhomogeneous patterns arise. 
The first one is related to commutative algebra, the second one is on algebraic geometry, and the third one is related to finite geometry.

\subsection{Semigroups with maximal embedding dimension}

\emph{Minimal generators} of a numerical semigroup are those elements that can not be obtained as the sum of any other two nonzero elements of the semigroup.  Equivalently, $x\in\Lambda$ is a minimal generator of the semigroup $\Lambda$ if and only if $\Lambda\setminus\{x\}$ is still a numerical semigroup. The number of minimal  generators (usually referred to as the \emph{embedding dimension}) is bounded by the multiplicity. Those numerical semigroups for which the number of minimal generators equals the multiplicity are said to be of {\it maximal embedding dimension} (MED). \
These semigroups also have other ``maximal'' properties as explained in  \cite{BDF} and \cite[Chapter 2]{R-GS:libro}.

Maximal embedding dimension numerical semigroups are characterized by the fact that for any two nonzero elements $x,y$ of the semigroup, one has that $x+y-m$ belongs to the semigroup where $m$ is its multiplicity.

This example, for a fixed $m$ (multiplicity) is related to the nonhomogeneous pattern $x_1+x_2-m$. 
From this, it easily follows that every Arf numerical semigroup has maximal embedding dimension.

\subsection{The Geil-Matsumoto bound}
An important problem of algebraic coding theory is upper bounding the maximum number of places of degree one of function fields.
The well known Hasse-Weil bound 
as well as Serre's improvement 
($q+1+g \lfloor 2\sqrt{q}\rfloor$)
use only the genus $g$ of the function field and the field size $q$.
Geil and Matsumoto give in \cite{GeilMatsumoto} a bound 
in terms of the field size and the Weierstrass semigroup $\Lambda$
of a rational place (that is, the set of pole orders of rational functions having only poles in that place). It is 
$\#(\Lambda\setminus\cup_{\lambda\in\Lambda\setminus\{0\}}(q\lambda_i+\Lambda))+1.$
It is a neat formula although it is not closed and it may be computationally hard to calculate. 
Lewittes' bound \cite{Lewittes} preceded the Geil-Matsumoto bound and it
only considers, apart from the field size, the multiplicity $m$ of the numerical semigroup. 
It is $1+qm$. It can be derived from the Geil-Matsumoto bound and so it is weaker. The obvious advantadge of Lewittes' bound with respect to the Geil-Matsumoto bound is that Lewittes' bound is very simple to compute. 
Furthermore, the results by Beelen and Ruano in \cite{BeelenRuano} allow bounding the number of rational places with nonzero coordinates by
$\#(\Lambda\setminus\cup_{\lambda\in\Lambda\setminus\{0\}}((q-1)\lambda_i+\Lambda))+1.$

It is proved in \cite{BA-VO} that the Geil-Matsumoto bound and the Lewittes' bound coincide if and only if
$qx-qm\in\Lambda$ for all $x\in\Lambda\setminus\{0\}$, where $m$ is the multiplicity of $\Lambda$. 
Similarly, it can be proved that Beelen-Ruano's bound on the number of 
rational places with nonzero coordinates equals $1+(q-1)m$ if and only if 
$(q-1)x-(q-1)m\in\Lambda$ for all $x\in\Lambda\setminus\{0\}$.

These examples, for a fixed $q$ (field size) and a fixed $m$ (multiplicity),
are related respectively to the nonhomogeneous patterns $qx_1-qm$
and $(q-1)x_1-(q-1)m$.

\begin{remark}
 Let $k$ be a positive integer, and let $\Lambda$ be a numerical semigroup with multiplicity $m$. 
%
Let $x$ and $y$ be two integers such that $kx-km, ky-km\in \Lambda$. Then $k(x+y)-km=(kx-km)+(ky-km)+km\in \Lambda$.
Consequently the following conditions are equivalent:
\begin{enumerate}[(a)]
 \item for every $x\in \Lambda\setminus\{0\}$, $kx-km\in \Lambda$,
 \item for every minimal generator $x$ of $\Lambda$, $kx-km\in \Lambda$.
\end{enumerate}
Set 
 \[\frac{\Lambda}k=\{ x\in \mathbb Z ~|~ kx\in \Lambda\},  
 \]
which is also a numerical semigroup (see \cite[Chapter 5]{R-GS:libro}). Then $kx-km\in \Lambda$ if and only if $x-m\in \Lambda/k$, or equivalently, $x\in m+\Lambda/k$. Hence the above conditions are also equivalent to
\begin{enumerate}[(c)]
\item  $\Lambda\setminus\{0\}\subseteq m+\Lambda/k$.
\end{enumerate}
In particular,  if $k\in \Lambda\setminus\{0\}$, then $\Lambda/k=\mathbb N_0$. Trivially $\Lambda\setminus\{0\}\subseteq m+\mathbb N_0$. 
\end{remark} 
%

\subsection{Combinatorial configurations}
A $(v,b,r,k)$-combinatorial configuration is an incidence structure with a set of $v$ points and a set of $b$ lines such that each line contains $k$ points, each point is contained in $r$ lines, and any two distinct lines are incident with at most one point or, equivalently, any two distinct points coincide in  at most one line. It is easy to prove that if a $(v,b,r,k)$-configuration exists, then necessarily $vr=bk$ and so, there exists an integer $d$ such that 
$(v,b,r,k)=(d\frac{k}{\gcd(r,k)},d\frac{r}{\gcd(r,k)},r,k)$.
For a fixed pair $r,k$, the set $D_{r,k}$ of all integers $d$ such that there exists a $(d\frac{k}{\gcd(r,k)},d\frac{r}{\gcd(r,k)},r,k)$-configuration is a numerical semigroup \cite{BA-S}.
It is proved in \cite{SB:CC_patterns} that $D_{r,k}$ satisfies the nonhomogeneous
pattern $x_1+x_2-n$ for any $n\in\{1,\dots,\gcd(r,k)\}$. See other related results in \cite{Grunbaum}.

\section{Nonhomogeneous patterns}
Here by a \emph{pattern} we will mean a linear polynomial 
with nonzero integer coefficients in $x_1,\dots,x_n$ and eventually a nonzero integer constant term. We will say that $n$ is the length of the pattern.
Homogeneous patterns were first introduced and studied in \cite{BA-GS}.
In that paper a semigroup was said to {\it admit} a pattern 
$p(x_1,\dots,x_n)$ if for every $n$ elements $s_1,\ldots,s_n$ in $\Lambda$ with $s_1 \geq s_2 \geq \cdots \geq s_n$ the integer $p(s_1,\dots,s_n)$ belonged to $\Lambda$. 
For a nonhomogeneous pattern with an integer nonzero constant term, it seems reasonable that the condition for a semigroup to admit it considers only nonzero elements of the semigroup. That is, we will say that a numerical semigroup admits a nonhomogeneous pattern 
$p(x_1,\dots,x_n)$ if for every $n$ {\it nonzero} elements $s_1,\ldots,s_n$ in $\Lambda$ with $s_1 \geq s_2 \geq \cdots \geq s_n$ the integer $p(s_1,\dots,s_n)$ belongs to $\Lambda$. 

We denote by $\mathcal{S}(p)$ the set of all numerical semigroups admitting $p$.

For the case of homogeneous patterns it was proved in \cite{BA-GS} that 
the following conditions are equivalent for a pattern $p=\sum_{i=1}^na_ix_i$:
\begin{itemize}
\item[(a)] $\mathcal{S}(p)\neq \emptyset$,
\item[(b)] $\mathbb N_0\in\mathcal{S}(p)$,
\item[(c)] $\sum_{i=1}^{j}a_i \geq 0$ for all $j\leq n$.
\end{itemize}

Here we will prove an equivalent result for nonhomogeneous patterns.
When dealing with nonhomogeneous 
patterns, the role that ${\mathbb N}_0$ played for homogeneous patterns will be played by an ordinary semigroup, that is, a semigroup of the form $\{0\}\cup (m+{\mathbb N}_0)$ for some integer $m$. This semigroup is represented by $\{0,m,\to\}$.

For a given pattern $p=\sum_{i=1}^n a_ix_i+a_0$ and for all $j\leq n$, considering the partial sums 
\begin{equation}
\label{eq:sigmas}
\sigma_{j}=\sum_{i=1}^{j}a_i
\end{equation}
will be useful for the formulation and proof of the following technical results.

\begin{lemma}\label{l:sums-neg}
 Let $p=\sum_{i=1}^n a_ix_i+a_0$ be a nonhomogeneous pattern such that $\sigma_j<0$ for some $j\in\{1,\ldots, n\}$. Then $\mathcal S(p)=\emptyset$.
\end{lemma}
\begin{proof}
For any numerical semigroup, let $m$ be its multiplicity 
and take any positive integer $l$ in the semigroup larger than $\frac{\sum_{k=j+1}^na_km+a_0}{-\sigma_j}$.
Then $\sum_{i=1}^{j}a_il+\sum_{i=j+1}^na_im+a_0<
\sum_{i=1}^{j}a_i\frac{\sum_{k=j+1}^na_km+a_0}{-\sum_{k=1}^{j}a_k}+\sum_{i=j+1}^na_im+a_0=0$, and so it does not belong to the semigroup. 
\end{proof}

\begin{lemma}\label{l:sigman-neg}
 Let $p=\sum_{i=1}^n a_ix_i+a_0$ be a nonhomogeneous pattern such that $\sigma_n\le 0$ and $a_0<0$. Then $\mathcal S(p)=\emptyset$.
\end{lemma}
\begin{proof}
Let $\Lambda$ be a numerical semigroup and $m$ be its multiplicity. Then $p(m,\ldots,m)= \sigma_n m + a_0  < 0$, so no numerical semigroup can admit $p$.
\end{proof}

Let $\Lambda$ be a numerical semigroup and let $x$ be a nonzero element of $\Lambda$. The Ap\'ery set of $x$ in $\Lambda$ is 
\[ \mathrm{Ap}(\Lambda,x)=\{ s\in \Lambda~|~ x-s\not\in \Lambda\}.\]
This set has exactly $x$ elements, one for each congruent class modulo $x$ (\cite[Chapter 1]{R-GS:libro}).

\begin{lemma}\label{l:sigman-uno-a0-menor-menos-uno}
 Let $p=\sum_{i=1}^n a_ix_i+a_0$ be a nonhomogeneous pattern such that $\sigma_n= 1$ and $a_0<0$. Then $\mathcal S(p)\subseteq \{\mathbb N_0\}$.
\end{lemma}
\begin{proof}
Let $\Lambda$ be a numerical semigroup admitting $p$ and $m$ be its multiplicity. In this setting, $p(m,\ldots,m)= m + a_0$. Hence $p(m,\ldots,m)\in \Lambda$ forces $a_0=-m$. If $\Lambda\neq \mathbb N_0$, then $m\neq 1$. Take an element $x\in \mathrm{Ap}(\Lambda,m)\setminus\{0\}$ (this set is not empty since $m>1$). Then $p(x,\ldots,x)=x - m \notin \Lambda$, and so $\Lambda$ does not admit $p$, a contradiction. 
\end{proof}

\begin{lemma}
\label{l:sums}
If the nonhomogeneous pattern $p=\sum_{i=1}^n a_ix_i+a_0$, $a_0\neq 0$, is admitted at least by one semigroup other than $\mathbb N_0$, then $\sigma_j\geq 0$, for all  $j\leq n$, and either $a_0\geq 0$ or $\sigma_n>1$. 
\end{lemma}
\begin{proof}
Suppose first that there exists $j\leq n$ such that $\sigma_j<0$. Lemma \ref{l:sums-neg} asserts that ${\mathcal S}(p)=\emptyset$, a contradiction.

Assume now that $a_0< 0$ and $\sigma_n\le 1$. Then Lemmas \ref{l:sigman-neg} and \ref{l:sigman-uno-a0-menor-menos-uno} state that there is no numerical semigroup other than $\mathbb N_0$ admitting $p$.
\end{proof}

Let $s_1,\dots,s_n$ be a nonincreasing sequence of nonzero elements of 
a semigroup $\Lambda$.
Define 
\begin{equation}
\label{eq:deltas}
\delta_j=\left\{\begin{array}{ll}
s_j-s_{j+1} &\mbox{if }1\leq j\leq n-1,\\
s_n-m &\mbox{if }j=n.
\end{array}\right.
\end{equation}
Then $\delta_i\geq 0$ for all $1\leq i\leq n$.
In particular, $s_i=\sum_{j=i}^{n-1}(s_{j}-s_{j+1})+(s_n-m)+m=
\sum_{j=i}^{n}\delta_{j}+m$. 
Then,
\begin{eqnarray}
p(s_1,\dots,s_n)&=&\sum_{i=1}^n a_is_i+a_0 =\sum_{i=1}^n a_i\left(\sum_{j=i}^{n}\delta_j+m\right)
+a_0\nonumber\\ &=&
\sum_{i=1}^n a_i\sum_{j=i}^{n}\delta_j+\sigma_nm+a_0 =\sum_{j=1}^{n}\sum_{i=1}^{j}a_i\delta_j+\sigma_nm+a_0\nonumber\\&=&
\sum_{j=1}^{n}\sigma_j\delta_j+\sigma_nm+a_0.\label{semi}
\end{eqnarray}


\begin{lemma}
\label{l:sufcondforordinary}
Suppose that the nonhomogeneous pattern $p=\sum_{i=1}^n a_ix_i+a_0$
satisfies
$\sigma_{j}=\sum_{i=1}^{j}a_i\geq 0$ for all $j\leq n$
and either $a_0\geq 0$ or $\sigma_n>1$.
 Let $m$ be any positive integer satisfying
$$\left\{\begin{array}{ll}
m\geq -\frac{a_0}{\sigma_{n}-1}&\mbox{ if }\sigma_n>1,
\\
m\geq 0 &\mbox{ if }\sigma_n=1 (\mbox{and so }a_0\geq 0),\\
m\leq a_0&\mbox{ if }\sigma_n=0.
\end{array}\right.
$$
Then the ordinary semigroup $\{0,m,\to\}$ admits $p$.
\end{lemma}
\begin{proof}
Suppose that
$s_1,\dots,s_n$ is a nonincreasing sequence of nonzero elements of
$\{0,m,\to\}$.
By \eqref{semi} and the hypothesis that $\sigma_j\geq 0$,
\[p(s_1,\dots,s_n)= \sum_{j=1}^{n}\sigma_j\delta_j+\sigma_nm+a_0 \geq \sigma_nm+a_0.\]

Now, if $\sigma_n=0$, then $p(s_1,\dots,s_n)\geq a_0$. So, in this case,
if $m$ satisfies the hypothesis $m\leq a_0$,
then $p(s_1,\dots,s_n)\in\{0,m,\to\}$.
If $\sigma_n=1$ (and so $a_0\geq 0$), then $p(s_1,\dots,s_n)\geq
m+a_0\geq m$, so, again,
$p(s_1,\dots,s_n)\in\{0,m,\to\}$.
Otherwise if $\sigma_n>1$,

\begin{eqnarray}
p(s_1,\dots,s_n)
&\geq&\sigma_nm+a_0 =  \left(\sigma_n-1\right)m+a_0+m\nonumber\\
&\geq &\left(\sigma_n-1\right)(-\frac{a_0}{\sigma_{n}-1})+a_0+m = m.\nonumber
\end{eqnarray}

So, $p(s_1,\dots,s_n)\in\{0,m,\to\}$.
\end{proof}

\begin{remark}
Let  $p=\sum_{i=1}^n a_ix_i+a_0$ be a nonhomogeneous pattern
such that 
$\sigma_{j}=\sum_{i=1}^{j}a_i\geq 0$ for all $j\leq n$, $a_0=-1$, and $\sigma_n=1$. Then Lemma~\ref{l:sigman-uno-a0-menor-menos-uno} asserts that $\mathcal S(p)\subseteq \{\mathbb N_0\}$. Also from \eqref{semi}, we obtain that for every 
nonincreasing sequence $s_1,\dots,s_n$ of positive integers, $p(s_1,\ldots, s_n)= \sum_{j=1}^{n}\sigma_j\delta_j+m-1$, which is a nonnegative integer. Hence $\mathcal S(p)= \{\mathbb N_0\}$.
\end{remark}

From the previous lemmas we obtain the following theorem.

\begin{theorem}
\label{t:adm}
The next conditions are equivalent for a pattern
$p=\sum_{i=1}^na_ix_i+a_0$, with $a_0\neq 0$, where
$\sigma_j=\sum_{i=1}^ja_i$:
\begin{enumerate}
\item[(a)] $\mathcal{S}(p)\not\subseteq \{\mathbb N_0\}$,
\item[(b)] either $a_0\geq 0$ or $\sigma_n>1$, and $\sigma_j\geq 0$ for all $j\leq n$.
\end{enumerate}
If any of these two conditions hold, then $\{0,m,\to\}\in\mathcal{S}(p)$ for all $m$ satisfying
\begin{equation}\label{admmult}\left\{\begin{array}{ll}
m\geq -\frac{a_0}{\sigma_{n}-1}&\mbox{ if }\sigma_n>1,
\\
m\geq 0 &\mbox{ if }\sigma_n=1 \mbox{ and so }a_0\geq 0,\\
m\leq a_0&\mbox{ if }\sigma_n=0.
\end{array}\right.
\end{equation}

\end{theorem}

The patterns satisfying the conditions in Theorem~\ref{t:adm} are called {\it admissible} patterns.
Given one such pattern $p$, the multiplicities $m$ satisfying \eqref{admmult}
are called $p$-{\it admissible} multiplicities. 

\section{Patterns involving the multiplicity}
Notice that both the pattern associated to the Geil-Matsumoto bound ($qx_1-qm$) and the pattern associated to the maximal embedding dimension semigroups ($x_1+x_2-m$) involve in their constant parameter the multiplicity of the semigroup. 
The patterns whose constant term is an integer multiple of the multiplicity can be seen as an intermediate class between homogeneous and nonhomogeneous patterns.

Here we are interested in the semigroups not only admitting the pattern but also having the desired multiplicity. Let ${\mathcal S}_m(p)$ be the set of numerical semigroups with multiplicity $m$ admitting the pattern $p$.
The first result we would like to analyze is whether, parallelizing the previous results, ${\mathcal S}_m(p)\neq\emptyset$ is equivalent to $\{0,m,\to\}\in{\mathcal S}_m(p)$.
But we can see that this is not true in general. Indeed, the pattern related to the Geil-Matsumoto bound (that is, $qx_1-qm$) gives a counterexample. Just take $q=2$ and $m=3$.
In this case $\{0,3,\to\}\not\in{\mathcal S}_3(2x_1-6)$ because evaluating the pattern $2x_1-6$ at $s_1=4$ gives $2\not\in\{0,3,\to\}$. However, 
${\mathcal S}_3(2x_1-6)\neq\emptyset$ because for instance the semigroup $\{0,3,6,\to\}$ belongs to ${\mathcal S}_3(2x_1-6)$.
Nevertheless, we show that for some particular cases, we can find results similar to the ones in the previous sections.

\begin{theorem}
\label{6:theorem:adpat}
Let $p = \sum^n_{i=1} a_ix_i +k m$ be  a nonhomogeneous pattern, with $m>1$ and $k$ a nonzero integer. Set $\sigma_j=\sum^{j}_{i=1}a_i$. Assume that either $k=-1$ or that there exists $j\in \{1,\ldots,n\}$ such that $\sigma_j=1$. The following conditions are equivalent:
\begin{enumerate}[(a)]
\item there exists a numerical semigroup of multiplicity $m$ that admits $p$,
\item $\{0,m,\to\}$ admits $p$,
\item $\left \{ \begin{array}{l} \sigma_j \geq 0 \mbox{ for all } j \leq n,\\ \sigma_n +k\geq 1. \end{array}\right .$
\end{enumerate}
\end{theorem}
\begin{proof}
\emph{(a) implies (c).} Let $\Lambda$ be a numerical semigroup with multiplicity $m$ admitting $p$ (as $m>1$, $\Lambda\neq\mathbb N_0$). In light of Theorem \ref{t:adm}, $\sigma_j\ge 0$ for all $j\le n$. So, it remains to prove that $\sigma_n+k\ge 1$. We distinguish two cases: $k>0$ and $k<0$. For the first case, the assertion follows trivially since $\sigma_n\ge 0$ and $k\ge 1$. 

If $k<0$, by Condition (b) in Theorem \ref{t:adm}, we get $\sigma_n>1$. For $k=-1$ we have $\sigma_n+k\ge 1$, and for $k<-1$, by hypothesis, there must be $j\in\{1,\ldots,n\}$ such that $\sigma_j=1$. Let $\overline{\sigma}_j=\sum_{i=j+1}^n a_i$. Assume to the contrary that $\sigma_n+k\le 0$. Then $\sigma_j+\overline{\sigma_j}+k=-t$ for some nonnegative integer $t$, and $\overline{\sigma}_j+k=-t-1=-(t+1)<0$. Let $x$ be a nonzero element of $\mathrm{Ap}(\Lambda,(t+1)m)$. 
Set $s_1=x,\ldots, s_j=x, s_{j+1}=m,\ldots, s_n=m$. Then $p(s_1,\ldots,s_n)= \sigma_j x+\overline{\sigma}_j m + km = x-(t+1)m\not \in \Lambda$, a contradiction.

\emph{(c) implies (b).} By hypothesis $\sigma_n\ge 0$. We distinguish three cases.
\begin{itemize}
\item If $\sigma_n>1$, then $\sigma_n-1>0$, and consequently $m\ge \frac{-k m}{\sigma_n-1}$ if and only if $\sigma_n-1\ge -k$, that is, $\sigma_n+k\ge 1$. By Theorem \ref{t:adm}, $\{0,m,\to\}\in \mathcal S(p)$.

\item If $\sigma_n=1$, then $\sigma_n+k\ge 1$ forces $km\ge 0$. As $m\ge 0$, Theorem \ref{t:adm} ensures that $\{0,m,\to\}$ admits $p$.
 
\item For $\sigma_n=0$, the condition $\sigma_n+k\ge 1$, implies that $k\ge 1$. Hence $km\ge 0$, and by Theorem \ref{t:adm}  we obtain that $\{0,m,\to\}\in \mathcal S(p)$, because trivially $m\le km$. 
\end{itemize}

\emph{(b) implies (a)}. Trivial.
\end{proof}

\mut{
Following the notation in \cite{BA-GS} we can define a pattern $p=\sum_{i=1}^na_ix_i+a_0$ to be {\it premonic} 
if $\sum_{i=1}^{j}a_i=1$ for some $j\leq n$. In particular, monic patterns are premonic. 
The patterns for maximal embedding dimension numerical semigroups as well as the patterns associated to combinatorial configurations are premonic, while the pattern associated to the Geil-Matsumoto bound is not premonic unless $q=1$.
}

\section{Nonhomogeneous Frobenius varieties}

A \emph{Frobenius variety} is a nonempty family $\mathcal V$ of numerical semigroups such that
\begin{enumerate}
\item if $\Lambda_1,\Lambda_2\in \mathcal V$, then $\Lambda_1\cap \Lambda_2\in \mathcal V$,

\item if $\Lambda\in \mathcal V$, $\Lambda\neq \mathbb N_0$, then $\Lambda\cup \{\mathrm F(\Lambda)\}\in \mathcal V$.
\end{enumerate}

The families of Arf, saturated, system proportionally modular numerical semigroups, and those admitting an homogeneous admissible pattern are Frobenius varieties (see \cite{RGGB}, \cite{saturated}, \cite{spm} and \cite{BA-GS}, respectively). The class of system proportionally modular numerical semigroups coincides with the set of numerical semigroups having a Toms decomposition \cite{ns-toms}, and thus every numerical semigroup in this family can be realized as the positive cone of the $K_0$-group of a $C^*$-algebra \cite{toms}. 

Frobenius varieties were precisely introduced by Rosales in \cite{variedades-frobenius} because he observed that there was a common factor in \cite{RGGB,saturated,spm,BA-GS}: some of the proofs were based on the fact that these families were closed under intersections and the adjoin of the Frobenius number. Also due to this fact, it was possible to define minimal generating systems with respect to any of these families that are, in general, smaller than classical minimal generating systems (which are obtained by simply considering the Frobenius variety of all numerical semigroups). As in the classical sense, a minimal generator of a numerical semigroup $\Lambda$ in a Frobenius variety is an element $x\in \Lambda$ such that $\Lambda\setminus\{x\}$ is also in the Frobenius variety. 
This allows to arrange the semigroups in a Frobenius variety in a tree rooted in $\mathbb N_0$, and consequently theoretically construct all numerical semigroups in the variety up to a given genus.

We now modify slightly the definition of Frobenius variety mainly inspired in \cite{MED}. The proofs are similar to the classical case, indeed we will follow the sequence of arguments given in \cite[Sections 6.4 and 6.5]{R-GS:libro}.

Let $m$ be a positive integer, and let $\mathcal V$ be a set of numerical semigroups with multiplicity $m$. We say that $\mathcal V$ is a
\emph{nonhomogeneous Frobenius variety of multiplicity $m$} or \emph{$m$-variety} for short if 

\begin{enumerate}[({V}1)]
\item for every $\Lambda_1,\Lambda_2\in \mathcal V$, $\Lambda_1\cap \Lambda_2$ is also in $\mathcal V$,
\item for every $\Lambda\in \mathcal V$, $\Lambda\neq \{0,m,\to\}$, $\Lambda\cup \mathrm F(\Lambda)\in \mathcal V$.
\end{enumerate}

Observe that according to this second condition, the semigroup $\{0,m,\to\}$ is in  $\mathcal V$. Also, in light of \cite[Proposition 3 and Lemma 10]{MED}, the set of maximal embedding dimension numerical semigoups with multiplicity $m$ is an $m$-variety.

A submonoid $M$ of $\mathbb N_0$ is a \emph{$\mathcal V$-monoid} if $M$ can be expressed as an intersection of elements of $\mathcal V$. For a set of integers $A$ larger than or equal to $m$, the \emph{$\mathcal V$-monoid generated} by $A$, denoted by $\mathcal V(A)$, is the intersection of all elements in $\mathcal V$ containing $A$. The condition $A\subseteq \{0,m,\to\}$ implies that $\mathcal V(A)$ is not empty, and thus it is indeed a submonoid of $\mathbb N_0$. For a $\mathcal V$-monoid $M$, we say that $A\subseteq M$ is a \emph{$\mathcal V$-generating system}, or that $A$ \emph{$\mathcal V$-generates} $M$, if $\mathcal V(A)=M$. In addition $A$ is a \emph{minimal} $\mathcal V$-generating system if no proper subset of $A$ $\mathcal V$-generates $M$.  Notice that $m$ is never in a minimal $\mathcal V$-generating system of any $\mathcal V$-monoid.

From now on, given a subset $A$ of ${\mathbb N}_0$, we will use the notation
$$\langle A\rangle=\left\{\sum_{i=1}^nk_ia_i:n\in{\mathbb N_0}, k_i\in{\mathbb N_0}, a_i\in A\mbox{ for all }i\in\{1,\dots,n\}\right\}.$$

\begin{remark}\label{propiedades-v-monoides}
The following facts are easy consequences of the definitions.
\begin{enumerate}[1)]
\item The intersection of $\mathcal V$-monoids is a $\mathcal V$-monoid.
\item Let $A$ and  $B$ be subsets of $\{0,m,\to\}$. If $A\subseteq B$, then $\mathcal V(A)\subseteq \mathcal V(B)$.
\item For every set $A$ of integers larger than or equal to $m$, $\mathcal V(\langle A\rangle)=\mathcal V(A)$. 
\item If $M$ is a $\mathcal V$-monoid, then $\mathcal V(M)=M$.
\end{enumerate}
These two last assertions imply that every $\mathcal V$-monoid admits a $\mathcal V$-generating system with finitely many elements.
\end{remark}

From this remark, the following characterization of minimal $\mathcal V$-generating systems can be proved easily (its proof is the same as \cite[Lemma 7.24]{R-GS:libro}).

\begin{lemma}\label{carac-v-sistemas-minimales}
Let $A\subseteq \{0,m,\to\}$ and $M=\mathcal V(A)$. The set $A$ is a minimal $\mathcal V$-generating system of $M$ if and only if $a\not\in \mathcal V(A\setminus\{a\})$ for all $a\in A$.
\end{lemma}

Next lemma is the key result to show that minimal $\mathcal V$-generating systems are unique. Its proof goes as that of \cite[Lemma 7.25]{R-GS:libro} with a slight modification.

\begin{lemma}\label{llave-para-unicidad-v-sistemas-minimales}
Let $A\subseteq \{0,m,\to\}$. If $x\in \mathcal V(A)$, then $x\in \mathcal V(\{a\in A~|~ a\le x\})$.
\end{lemma}
\begin{proof}
Assume to the contrary that $x\not\in \mathcal V(\{a\in A~|~ a\le x\})$. Notice that this forces $x\not\in \{0,m,\to\}$ and $x\not\in A$. From the definition of $\mathcal V(\{a\in A~|~ a\le x\})$, it follows that there exists $\Lambda\in \mathcal V$ containing $\{a\in A~|~ a\le x\}$ such that $x\not \in \Lambda$. As $m<x\le \mathrm F(\Lambda)$, by applying as many times as needed Condition V2, the set $\Lambda\cup\{0,x+1,\to\}$ is in $\mathcal V$. Clearly $A\subseteq \Lambda\cup\{0,x+1,\to\}$ and $x\not\in \Lambda\cup\{0,x+1,\to\}$, which implies $x\not\in \mathcal V(A)$, a contradiction.
\end{proof}

\begin{theorem}\label{unicidad-v-sistema-minimal}
Let $m$ be a positive integer and let $\mathcal V$ be an $m$-variety. Every $M$ monoid has a unique minimal $\mathcal V$-system of generators with finitely many elements.
\end{theorem}
\begin{proof}
The proof that minimal $\mathcal V$-generating systems are unique is the same as that of \cite[Theorem 7.26]{R-GS:libro}, where Remark \ref{propiedades-v-monoides} plays the role of \cite[Lemma 7.22]{R-GS:libro}, and Lemmas \ref{carac-v-sistemas-minimales} and \ref{llave-para-unicidad-v-sistemas-minimales} are the analogues to \cite[Lemmas 7.24 and 7.25]{R-GS:libro}, respectively. The finiteness condition is a consequence of the last paragraph of Remark \ref{propiedades-v-monoides}.
\end{proof}

An element in the unique minimal $\mathcal V$-generating system of a $\mathcal V$-monoid will be called a \emph{minimal $\mathcal V$-generator}. As a consequence of Theorem \ref{unicidad-v-sistema-minimal}, these elements can now be characterized as in a Frobenius variety. Actually, the proof of this description is exactly the same as \cite[Proposition 7.28]{R-GS:libro}. Notice that, as we already mentioned above, $m$ is not a minimal $\mathcal V$-generator for any monoid in an $m$-variety.

\begin{corollary}\label{caracterizacion-v-generador-minimal}
Let $M$ be a $\mathcal V$-monoid and let $x\in M$. The set $M\setminus\{x\}$ is a $\mathcal V$-monoid if and only if $x$ is a minimal $\mathcal V$-generator.
\end{corollary}

From this last result we easily obtain a slight modification of \cite[Corollary 7.29]{R-GS:libro}. The difference strives in the fact that for $\Lambda$ in a $m$-variety $\mathcal V$, $\Lambda\cup\{\mathrm F(\Lambda)\}$ is not in $\mathcal V$ for $\Lambda=\{0,m,\to\}$. As we explain next, this result is used to arrange the elements of $\mathcal V$ in a tree.

\begin{corollary}
Let $\Lambda$ be a $\mathcal V$-monoid. The following are equivalent.
\begin{enumerate}[1)]
\item $\Lambda=\Lambda'\cup \{\mathrm F(\Lambda')\}$ for some $\Lambda'\in \mathcal V$.
\item The minimal $\mathcal V$-generating system of $\Lambda$ contains an element larger than $\mathrm F(\Lambda)$.
\end{enumerate}
\end{corollary}
\begin{proof}
\emph{1) implies 2).} Clearly $\Lambda'=\Lambda\setminus \{\mathrm F(\Lambda')\}\in \mathcal V$, and by Corollary \ref{caracterizacion-v-generador-minimal} we deduce that $\mathrm F(\Lambda')$ is a minimal $\mathcal V$-generator of $\Lambda$. Notice that $\mathrm F(\Lambda)<\mathrm F(\Lambda')$.

\emph{2) implies 1).} If $x$ is a minimal $\mathcal V$-generator, then Corollary \ref{caracterizacion-v-generador-minimal} ensures that $\Lambda\setminus\{x\}\in \mathcal V$. If in addition  $x$ is larger than $\mathrm F(\Lambda)$, then $\mathrm F(\Lambda\setminus \{x\})=x$, and thus we can choose $\Lambda'=\Lambda\setminus\{x\}$.
\end{proof}

With these ingredients, for an $m$-variety $\mathcal V$, we can arrange all the elements of $\mathcal V$ in a tree rooted in $\{0,m,\to\}$. For a vertex $\Lambda\in\mathcal V$ with minimal $\mathcal V$-generating system $A$, its descendants are $\Lambda\setminus\{a\}$ for all $a\in A$ with $a\ge \mathrm F(\Lambda)$ (\cite[Theorem 7.30]{R-GS:libro}).  Observe also that from any vertex in the tree we can construct the path to $\{0,m,\to\}$ by applying as many times as required Condition V2. In \cite[Figure 1]{MED} the tree of maximal embedding dimension numerical semigroups with multiplicity 4 is shown up to genus 5. In view of \cite[Corollary 17]{MED} this tree has infinitely many elements. 

\section{Nonhomogeneous strongly admissible patterns}
We can now define nonhomogeneous strongly admissible patterns in a similar way as it was done in \cite{BA-GS}. Given an admissible pattern $p = \sum_{i=1}^n a_ix_i+a_0$, set $$p'=\left\{\begin{array}{lr}p - x_1 & \mbox{ if } a_1 > 1, \\ p(0,x_1,x_2,\ldots,x_{n-1}) & \mbox{ if } a_1=1.
\end{array} \right .$$ 
Observe that since $p$ is a pattern $a_1\neq 0$, and as we are choosing it to be admissible, by Lemmma \ref{l:sums-neg}, $a_1=\sigma_1\ge 0$.

Define for $p'$ the partial sums $\sigma'_j$ as in \eqref{eq:sigmas}.
A nonhomogeneous admissible pattern $p$ is said to be \emph{strongly} admissible if $\sigma'_j\geq 0$ for all possible $j$. 

Note that it can be the case that $p$ is strongly admissible although $p'$ is not admissible. As an example we can take the pattern $x_1+x_2-1$. In this case $p'=x_1-1$ is not admissible, but still, $p$ is considered to be strongly admissible.

We are going to prove that the set of numerical semigroups with given multiplicity $m$ admitting a strongly admissible pattern form an $m$-variety. To this end, we need to prove the following technical lemma that will also be used later.

\begin{lemma}
\label{l:k1}
Let $p$ be a strongly admissible pattern of length $n$ and let $m$ be a $p$-admissible multiplicity. Then for every sequence of integers $s_1\geq \cdots\geq s_n\geq m$, it holds that $p(s_1,\dots,s_n)\geq s_1\geq \cdots\geq s_n$.
\end{lemma}

\begin{proof}
Assume that $p=\sum_{i=1}^na_ix_i+a_0$. 
If $a_1>1$, then $\sigma_n-1=\sigma_{n}'\ge 0$. Hence $\sigma_n\ge 1$. Define $\delta_j$ as in \eqref{eq:deltas}.
By using that $\sigma_j'\ge 0$, in view of \eqref{semi}, we get  $p'(s_1,\dots,s_n)=\sum_{j=1}^n\sigma'_j\delta_j+\sigma'_nm+a_0\geq \sigma'_nm+a_0=(\sigma_n-1)m+a_0$. For $\sigma_n=1$, since $p$ is admissible, Theorem \ref{t:adm} says that $a_0\ge 0$, and consequently $p'(s_1,\ldots,s_n)\geq a_0\ge 0$. If $\sigma_n>1$,  since $m$ is a $p$-admissible multiplicity, $m\ge \frac{-a_0}{\sigma_n-1}$, which leads to  $p'(s_1,\dots,s_n)\geq 0$.
In both cases, $p(s_1,\dots,s_n)=s_1+p'(s_1,\dots,s_n)\geq s_1$.

Now assume that $a_1=1$. In this setting, $\sigma_n-1=\sigma_{n-1}'$, and the proof follows as in the case $a_1>1$.
\end{proof}

%


\begin{lemma}
\label{l:k2} Let $p=\sum_{i=1}^na_ix_i+a_0$ be a strongly admissible pattern, and let $m$ be a $p$-admissible multiplcicity.
\begin{enumerate}
\item
If a nonordinary numerical semigroup $\Lambda$ of multiplicity $m$ admits $p$, then so does $\Lambda\cup \{\mathrm F(\Lambda)\}$.
\item The intersection of two numerical semigroups of multiplicity $m$ admitting $p$ also has multiplicity $m$ and admits $p$. 
\end{enumerate}
\end{lemma}

\begin{proof}
The first statement follows from Lemma~\ref{l:k1} while the second 
statement is immediate from the definitions.
\end{proof}

\begin{theorem}
\begin{enumerate}
\item
Given a strongly admissible pattern $p$ and a $p$-admissible multiplicity $m$, the set of all semigroups with multiplicity $m$ admitting $p$ is an $m$-variety.
\item
Given a set of strongly admissible patterns $p_1,\dots,p_r$ and a multiplicity $m$ that is $p_i$-admissible for all $p_i\in\{p_1,\dots,p_r\}$, the set of all semigroups with multiplicity $m$ admitting simultaneously $p_1,\dots,p_r$ is an $m$-variety.
\end{enumerate}
\end{theorem}


\begin{example}
We would like to 
highlight the beauty of the patterns $x_1+x_2+1$ and $x_1+x_2-1$ and their relationship with the intervals of nongaps of a numerical semigroup.
Indeed, the semigroups admitting $x_1+x_2+1$ can be characterized by the fact that the maximum element in each interval of nongaps is a minimal generator. 
Similarly, the semigroups admitting $x_1+x_2-1$ can be characterized by the fact that the minimum element in each interval of nongaps is a minimal generator. Figure~\ref{fig:xym1} represents the (finite) tree of all numerical semigroups admitting the pattern $x_1+x_2-1$ with multiplicity five.

\begin{figure}

\caption{The tree of all numerical semigroups admitting the pattern $x_1+x_2-1$ with multiplicity five.} 
\label{fig:xym1}
\medskip

    \centerline{
      \xymatrix @R=1pc @C=1pc{
      & \langle 5,6,7,8,9\rangle & \\
      \langle 5, 7, 8, 9, 11\rangle\ar@{-}[ur] & \langle 5,6,8,9\rangle\ar@{-}[u] &  \langle 5,6,7,9\rangle\ar@{-}[ul] \\
      \langle 5, 8, 9, 11, 12\rangle\ar@{-}[u] & \langle 5,7,9,11,13\rangle\ar@{-}[ul] &  \langle 5,6,9,13\rangle\ar@{-}[ul] \\
      \langle 5, 8, 9, 12\rangle\ar@{-}[u] & \langle 5,9, 11, 12, 13\rangle\ar@{-}[ul] &\\
       & \langle 5,9, 11, 13, 17 \rangle\ar@{-}[u] & \langle 5,  9, 12, 13, 16 \rangle\ar@{-}[ul]\\
      & &\langle  5, 9, 13, 16, 17 \rangle\ar@{-}[u] \\
      & &\langle  5, 9, 13, 17, 21 \rangle\ar@{-}[u] \\
      }
      }
      
\end{figure}
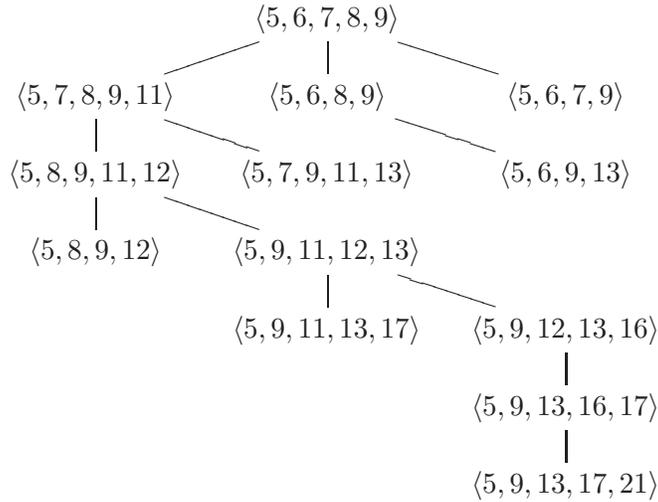
\end{example}

This example gives rise to a natural question. When is the tree of numerical semigroups with fixed multiplicity and admitting a strongly admissible pattern finite? The answer to this question is given in the following result.

\begin{theorem}
Let $p=\sum_{i=1}^na_ix_i+a_0$ be a strongly admissible pattern, and let $m$ be a $p$-admissible multiplicity. Then $\mathcal S_m(p)$ contains 
infinitely many numerical semigroups if and only if  $\gcd(m,a_0)\neq 1$.
\end{theorem}
\begin{proof}
\emph{Necessity.} Assume that $\gcd(m,a_0)=1$. Then $\gcd(m, p(m,\ldots,m))=1$, and consequently $\langle m,p(m,\ldots, m)\rangle$ is a numerical semigroup. If $\Lambda \in \mathcal S_m(p)$, then $\langle m,p(m,\ldots, m)\rangle \subseteq \Lambda$. Since there are finitely many numerical semigroups containing $\langle m,p(m,\ldots, m)\rangle$, we deduce that the cardinality of $\mathcal S_m(p)$ is finite.

\emph{Sufficiency.} Suppose now that $\gcd(m,a_0)=d\neq 1$, and let  $m_0=m/d$. Then for any $k\geq m$, the numerical semigroup $\Lambda_k=\{di: i \in{\mathbb N}_0, i\geq m_0\}\cup\{0,k,\to\}$ has multiplicity $m$ and admits $p$. Indeed, let $s_1,\dots,s_n$ be a nondecreasing sequence of nonzero elements of $\Lambda_k$. By Lemma~\ref{l:k1}, $p(s_1,\dots,s_n)$ is at least $s_n$. If $s_n\geq k$, then obviously $p(s_1,\dots,s_n)\in\Lambda_k$. Otherwise, $s_1,\dots,s_n$ are multiples of $d$ and so is $p(s_1,\dots,s_n)$. Now since $p(s_1,\dots,s_n)$ is a positive multiple of $d$ larger than or equal to $m$, $p(s_1,\dots,s_n)\in\Lambda_k$.
\end{proof}


\end{document}